\documentclass{amsart}
\usepackage{amssymb}
\usepackage{amsmath}
\usepackage{amsfonts}

\newtheorem{theorem}{Theorem}
\theoremstyle{plain}

\newtheorem{corollary}{Corollary}

\newtheorem{definition}{Definition}

\numberwithin{equation}{section}
\input{tcilatex}

\begin{document}
\title[Inclusion Theorems for Grand Lorentz Spaces]{Inclusion Theorems for
Grand Lorentz Spaces}
\author{Cihan UNAL}
\address{Sinop University Faculty of Arts and Sciences Department of
Mathematics \\
Sinop, TURKEY}
\email{cihanunal88@gmail.com}
\urladdr{}
\thanks{}
\author{Ismail AYDIN}
\address{Sinop University Faculty of Arts and Sciences Department of
Mathematics \\
Sinop, TURKEY}
\email{iaydin@sinop.edu.tr}
\urladdr{}
\date{}
\dedicatory{}
\thanks{}

\begin{abstract}
In this paper, we consider some inclusion theorems for grand Lorentz spaces $%
L^{p,q)}\left( X,\mu \right) $ and $\Lambda _{p),\omega }$ where $\mu $ is a
finite measure on $\left( X,\Sigma \right) .$ Moreover, we consider the
problem of the convergence of approximate identities in these spaces.
\end{abstract}

\maketitle

\section{Introduction}

Let $\left( X,\Sigma ,\mu \right) $ and $\left( X,\Sigma ,\nu \right) $ be
two finite measure spaces. It is known that $l^{p}\left( X\right) \subseteq
l^{q}\left( X\right) $ for $0<p\leq q\leq \infty .$ Subramanian \cite{Sub}
characterized all positive measures $\mu $ on $\left( X,\Sigma \right) $ for
which $L^{p}\left( \mu \right) \subseteq L^{q}\left( \mu \right) $ whenever $%
0<p\leq q\leq \infty $. Also, Romero \cite{Ro} investigated and developed
several results of \cite{Sub}. Moreover, Miamee \cite{Mi} obtained the more
general result as $L^{p}\left( \mu \right) \subseteq L^{q}\left( \nu \right) 
$ with respect to $\mu $ and $\nu .$ Aydin and Gurkanli \cite{Ayg} proved
some inclusion results for which $L^{p(.)}\left( \mu \right) \subseteq
L^{q(.)}\left( \nu \right) $. Moreover, these results was generalized by
Gurkanli\ \cite{Gu} and Kulak \cite{Ku} to the classical and variable
exponent Lorentz spaces. It is note that the classical Lorentz spaces are
introduced by Lorentz \cite{Lorentz1} and \cite{Lorentz2}.

In 1992, Iwaniec and Sbordone \cite{Iw} introduced grand Lebesgue spaces $%
L^{p)}\left( \Omega \right) $, $1<p<\infty $, on bounded sets $\Omega
\subset 
\mathbb{R}
^{d}.$ Also, Greco et al. \cite{Gr} obtained a generalized version $%
L^{p),\theta }\left( \Omega \right) $. Recently, these spaces have
intensively studied for various applications including harmonic analysis,
interpolation-extrapolation theory, analysis of PDE's etc. For more details,
we can refer \cite{Fio2}, \cite{Goga1}, \cite{Goga2}. Gurkanli \cite{Gur}
studied the inclusion $L^{p),\theta }\left( \mu \right) \subseteq
L^{q),\theta }\left( \nu \right) $ under some conditions for two different
measures $\mu $ and $\nu $ on $\left( X,\Sigma \right) $, and proved that $%
L^{p),\theta }\left( \mu \right) $ has no an approximate identities. Jain
and Kumari \cite{Jain} generalized the classical Lorentz spaces to grand
ones. They gave some basic properties of the grand Lorentz spaces and
investigated the boundedness of maximal operator for these spaces.

In this paper, we consider some properties of grand Lorentz spaces $%
L^{p,q)}\left( \Omega ,\mu \right) $ and $\Lambda _{p),\omega }$. Also, we
give the relationship between these spaces in sense to \cite{Jain}.
Moreover, we investigate the inclusion theorems for $L^{p,q)}\left( \Omega
,\mu \right) $ and $\Lambda _{p),\omega }$ and obtain more general results
than \cite{Gu}. Finally, we will discuss the approximate identities of $%
\Lambda _{p),\omega }$ regarding the boundedness of the maximal operator.

\section{\textbf{Notations and Preliminaries}}

The space $L_{loc}^{1}\left( 
\mathbb{R}
^{d}\right) $ is to be space of all measurable functions $f$ on $G$ such
that $f.\chi _{K}\in L^{1}\left( 
\mathbb{R}
^{d}\right) $ for any compact subset $K\subset 
\mathbb{R}
^{d}$ where $\chi _{K}$ is the characteristic function of $K$. A Banach
function space (shortly BF-space) on $%
\mathbb{R}
^{d}$ is a Banach space $\left( B,\left\Vert .\right\Vert _{B}\right) $ of
measurable functions which is continuously embedded into $L_{loc}^{1}\left( 
\mathbb{R}
^{d}\right) $, i.e. for any compact subset $K\subset 
\mathbb{R}
^{d}$ there exists some constant $C_{K}>0$ such that $\left\Vert f.\chi
_{K}\right\Vert _{L^{1}}\leq C_{K}.\left\Vert f\right\Vert _{B}$ for all $%
f\in B$.

\begin{definition}
(see \cite{Bloz})Let $\left( X,\Sigma ,\mu \right) $ be a finite measure
space and $f$ be a measurable function on $X.$ The distribution function of $%
f$ is defined by%
\begin{equation*}
\lambda _{f}\left( y\right) =\mu \left\{ x\in X:f\left( x\right) >y\right\}
\end{equation*}%
for all $y>0.$ Moreover, the rearrangement of $f$ is defined by%
\begin{equation*}
f^{\ast }\left( t\right) =\inf \left\{ y>0:\lambda _{f}\left( y\right) \leq
t\right\} =\sup \left\{ y>0:\lambda _{f}\left( y\right) >t\right\} ,\text{ \
\ \ \ }t>0
\end{equation*}%
where $\inf \emptyset =+\infty .$ Also, the average function of $f$ is
defined by%
\begin{equation*}
f^{\ast \ast }\left( t\right) =\frac{1}{t}\tint\limits_{0}^{t}f^{\ast
}\left( s\right) ds,\text{ \ \ \ \ }t>0.
\end{equation*}%
It is clear that $\lambda _{f}\left( .\right) $, $f^{\ast }\left( .\right) $
and $f^{\ast \ast }\left( .\right) $ are non-increasing and right continuous
on $\left( 0,\infty \right) $. Moreover, it is known that $f^{\ast }\leq
f^{\ast \ast }.$
\end{definition}

\begin{definition}
(see \cite{Jain})A measurable and locally integrable function $\omega
:X\longrightarrow \left( 0,\infty \right) $ is called a weight function. We
say that $\omega _{1}\prec \omega _{2}$ if only if there exists $c>0$ such
that $\omega _{1}(t)\leq c\omega _{2}(t)$ for all $t\in X$. Two weight
functions are called equivalent and written $\omega _{1}\sim \omega _{2}$,
if $\omega _{1}\prec \omega _{2}$ and $\omega _{2}\prec \omega _{1}$. The
classical Lorentz space $\Lambda _{p,\omega }$ consists of all measurable
functions such that%
\begin{equation*}
\left\Vert f\right\Vert _{\Lambda _{p,\omega }}=\left(
\tint\limits_{0}^{1}\left( f^{\ast }\left( t\right) \right) ^{p}\omega
(t)dt\right) ^{\frac{1}{p}}<\infty
\end{equation*}%
where $0<p<\infty $ and $\omega $ is a weight function, see \cite{Jain}.
\end{definition}

A special case of the space $\Lambda _{p,\omega },$ denoted by $%
L^{p,q}\left( X,\mu \right) $ (or shortly $L^{p,q}$), $0<p<\infty $ and $%
0<q\leq \infty ,$ consists of all measurable functions $f$ on $X$ such that%
\begin{equation*}
\left\Vert f\right\Vert _{p,q}=\left\{ 
\begin{array}{c}
\left( \frac{q}{p}\tint\limits_{0}^{\infty }t^{\frac{q}{p}-1}\left[ f^{\ast
}\left( t\right) \right] ^{q}dt\right) ^{\frac{1}{q}},\text{ \ \ \ }%
0<q<\infty \\ 
\underset{0<t<\infty }{\sup }\left[ t^{\frac{1}{p}}f^{\ast }\left( t\right) %
\right] ,\text{ \ \ \ }q=\infty%
\end{array}%
\right.
\end{equation*}%
is finite. The Lorentz space $L^{p,q}$ is a Banach space with respect to the
norm $\left\Vert .\right\Vert _{p,q}$. It is known that $L^{p}\left( \mu
\right) =L^{p,q}\left( X,\mu \right) $ whenever $p=q$ where $L^{p}\left( \mu
\right) $ is the Lebesgue space. Moreover, the space $L^{p,q}$ is a subspace
of $L^{p,s}$ for $q\leq s.$ Let $0<q\leq p\leq s\leq \infty .$ Then, we have%
\begin{equation*}
L^{p,q}\subset L^{p}\left( \mu \right) \subset L^{p,s}\subset L^{p,\infty }
\end{equation*}%
,see \cite{Bloz}. Denote%
\begin{equation*}
\left\Vert f\right\Vert _{p,q}^{\ast }=\left\{ 
\begin{array}{c}
\left( \frac{q}{p}\tint\limits_{0}^{\infty }\left[ t^{\frac{1}{p}}f^{\ast
\ast }\left( t\right) \right] ^{q}\frac{dt}{t}\right) ^{\frac{1}{q}},\text{
\ \ \ }0<q<\infty \\ 
\underset{0<t<\infty }{\sup }\left[ t^{\frac{1}{p}}f^{\ast \ast }\left(
t\right) \right] ,\text{ \ \ \ }q=\infty%
\end{array}%
\right. .
\end{equation*}%
Let $1<p<\infty $ and $1\leq q\leq \infty .$ Then, we have%
\begin{equation*}
\left\Vert f\right\Vert _{p,q}\leq \left\Vert f\right\Vert _{p,q}^{\ast
}\leq \frac{p}{p-1}\left\Vert f\right\Vert _{p,q}
\end{equation*}%
for every $f\in L^{p,q},$ see \cite{Bloz}.

\begin{definition}
(see \cite{Iw})The construction of the Lorentz space $L^{p,q}$ seems to be
inspired by the Lebesgue space $L^{p},$ where $f$ is replaced by its
non-increasing rearrangement. A generalization of Lebesgue space so-called
grand Lebesgue space denoted by $L^{p)},$ which for $1<p<\infty $ consists
of all measurable functions $f$ defined on $\left( 0,1\right) $ such that%
\begin{equation*}
\left\Vert f\right\Vert _{p)}=\sup_{0<\varepsilon <p-1}\left( \varepsilon
\tint\limits_{0}^{1}\left\vert f\left( x\right) \right\vert ^{p-\varepsilon
}dx\right) ^{\frac{1}{p-\varepsilon }}<\infty .
\end{equation*}%
This norm makes the space $L^{p)}$ is a BF-space. Moreover, we have $%
L^{p}\subset L^{p)}\subset L^{p-\varepsilon }$ for $\varepsilon \in \left(
0,p-1\right) ,$ see \cite{Iw}.
\end{definition}

Now, we are ready to give the definition of grand Lorentz space $L^{p,q)}.$

\begin{definition}
(see, \cite{Jain})Assume that $1<p,q\leq \infty $ and $X=\left( 0,1\right) .$
The grand Lorentz space $L^{p,q)}\left( X,\mu \right) $ (shortly $L^{p,q)}$
or $L^{p,q)}\left( \mu \right) $)is the space of all measurable functions $f$
on $X$ such that%
\begin{equation*}
\left\Vert f\right\Vert _{p,q)}=\left\{ 
\begin{array}{c}
\underset{0<\varepsilon <q-1}{\sup }\left( \frac{q}{p}\varepsilon
\tint\limits_{0}^{1}t^{\frac{q}{p}-1}\left[ f^{\ast }\left( t\right) \right]
^{q-\varepsilon }dt\right) ^{\frac{1}{q-\varepsilon }},\text{ \ \ \ }%
1<q<\infty \\ 
\underset{0<t<1}{\sup }\left[ t^{\frac{1}{p}}f^{\ast }\left( t\right) \right]
,\text{ \ \ \ }q=\infty%
\end{array}%
\right.
\end{equation*}%
is finite. It is note that these spaces coincide with the grand Lebesgue
spaces $L^{p)}$ whenever $q=p$. Also, the space $L^{p,q)}$ is a BF-space
under the condition $1<q<p<\infty .$
\end{definition}

Throughout this paper, we will accept that $\left( X,\Sigma \right) $ is a $%
\sigma $-finite measurable space where $X=\left( 0,1\right) $. Moreover, we
will assume that $1<p,q,r,s\leq \infty .$ Also, if two measures $\mu $ and $%
\nu $ are absolutely continuous with respect to each other (denoted by $\mu
\ll \nu $ and $\nu \ll \mu $), then we denote it by the symbol $\mu \approx
\nu .$ We shall throughout denote $W\left( t\right)
=\dint\limits_{0}^{t}\omega \left( s\right) ds$ and $V\left( t\right)
=\dint\limits_{0}^{t}\vartheta \left( s\right) ds$ for $t>0.$

\section{Inclusion Theorems of Grand Lorentz Spaces $L^{p,q)}$}

In this section, we will discuss and investigate the existence of the
inclusion $L^{p,q)}\left( X,\mu \right) \subseteq L^{r,s)}\left( X,\nu
\right) $ where $\mu $ and $\nu $ are different measures.

\begin{theorem}
Let $1<q<p.$ If $\left( f_{n}\right) _{n\in 
\mathbb{N}
}$ convergences to $f$ in $L^{p,q)}\left( X,\mu \right) $ then $\left(
f_{n}\right) _{n\in 
\mathbb{N}
}$ convergences to $f$ in measure.
\end{theorem}

\begin{proof}
Let $\left( f_{n}\right) _{n\in 
\mathbb{N}
}$ convergences to $f$ in $L^{p,q)}\left( X,\mu \right) .$ Then, for $%
\varepsilon \in \left( 0,q-1\right) $ we have%
\begin{equation}
\left( \frac{q}{p}\varepsilon \tint\limits_{0}^{1}t^{\frac{q}{p}-1}\left(
\left( f_{n}-f\right) ^{\ast }\left( t\right) \right) ^{q-\varepsilon }d\mu
\left( t\right) \right) ^{\frac{1}{q-\varepsilon }}\leq \left\Vert
f_{n}-f\right\Vert _{p,q),\mu }\longrightarrow 0  \label{BF1}
\end{equation}%
as $n\longrightarrow \infty .$ Since the space $L^{p,q)}\left( X,\mu \right) 
$ is a BF-space, we have%
\begin{equation}
\dint\limits_{X}\left( f_{n}\left( x\right) -f\left( x\right) \right) dx\leq
C_{X}\left\Vert f_{n}-f\right\Vert _{p,q),\mu }.  \label{BF2}
\end{equation}%
By (\ref{BF1}) and (\ref{BF2}), we get that $\left( f_{n}\right) _{n\in 
\mathbb{N}
}$ convergences to $f$ in $L^{1}\left( X\right) .$ This completes the proof.
\end{proof}

\begin{theorem}
\label{equivaltheo}Assume that $1<p,r<\infty $ and $1<q,s\leq \infty $. Then
the space $L^{p,q)}\left( X,\mu \right) $ is a subspace of $L^{r,s)}\left(
X,\nu \right) $ in sense of equivalence classes if and only if $\mu \approx
\nu $ and the space $L^{p,q)}\left( X,\mu \right) $ is a subspace of $%
L^{r,s)}\left( X,\nu \right) $ in sense of individual functions.
\end{theorem}

\begin{proof}
The sufficient condition of the theorem is clear. Now, we assume that $%
L^{p,q)}\left( X,\mu \right) \subseteq L^{r,s)}\left( X,\nu \right) $ holds
in the sense of equivalence classes. Let $f\in L^{p,q)}\left( X,\mu \right) $
be any individual function. This implies $f\in L^{p,q)}\left( X,\mu \right) $
in the sense of equivalent classes thus $f\in L^{r,s)}\left( X,\nu \right) $
in the sense of equivalent classes by the assumption. Hence we have $f\in
L^{r,s)}\left( X,\nu \right) $ in the sense of individual functions. This
shows that%
\begin{equation*}
L^{p,q)}\left( X,\mu \right) \subseteq L^{r,s)}\left( X,\nu \right)
\end{equation*}%
in sense of individual functions. Now, let us take any $A\in \Sigma $ with $%
\mu \left( A\right) =0.$ If $\chi _{A}$ is the characteristic function of $A$
then $\chi _{A}=0$ $\mu $-almost everywhere. Also, the rearrangement of $%
\chi _{A}$ is%
\begin{equation*}
\chi _{A}^{\ast }\left( t\right) =\left\{ 
\begin{array}{c}
1,\text{ \ \ \ }0<t<\mu \left( A\right) \\ 
0,\text{ \ \ \ }t\geq \mu \left( A\right)%
\end{array}%
\right. .
\end{equation*}%
Since $\mu \left( A\right) \leq 1$ and the function $\varepsilon ^{\frac{1}{%
q-\varepsilon }}$ is increasing in $\varepsilon \in \left( 0,q-1\right) $,\
we get%
\begin{eqnarray}
\left\Vert \chi _{A}\right\Vert _{p,q)} &=&\underset{0<\varepsilon <q-1}{%
\sup }\left( \frac{q}{p}\varepsilon \tint\limits_{0}^{1}t^{\frac{q}{p}-1}%
\left[ \chi _{A}^{\ast }\left( t\right) \right] ^{q-\varepsilon }dt\right) ^{%
\frac{1}{q-\varepsilon }}  \notag \\
&=&\underset{0<\varepsilon <q-1}{\sup }\left( \frac{q}{p}\varepsilon
\tint\limits_{0}^{\mu \left( A\right) }t^{\frac{q}{p}-1}\left[ \chi
_{A}^{\ast }\left( t\right) \right] ^{q-\varepsilon }dt\right) ^{\frac{1}{%
q-\varepsilon }}  \notag \\
&=&\underset{0<\varepsilon <q-1}{\sup }\left( \varepsilon \left( \mu \left(
A\right) \right) ^{\frac{q}{p}}\right) ^{\frac{1}{q-\varepsilon }}=\left(
q-1\right) \left( \mu \left( A\right) \right) ^{\frac{1}{p}}=0.
\label{indivi1}
\end{eqnarray}%
Now, assume that $1<p<\infty $ and $q=\infty .$ This yields%
\begin{equation}
\left\Vert \chi _{A}\right\Vert _{p,\infty ),\mu }=\underset{0<t<1}{\sup }%
\left[ t^{\frac{1}{p}}\chi _{A}^{\ast }\left( t\right) \right] =\mu \left(
A\right) =0.  \label{indivi2}
\end{equation}%
By (\ref{indivi1}) and (\ref{indivi2}), we have $\chi _{A}\in L^{p,q)}\left(
X,\mu \right) $ for $1<p<\infty $ and $1<q\leq \infty .$ Hence, the
characteristic function of $A$ is in the equivalent classes of $0\in
L^{p,q)}\left( X,\mu \right) $. Since the equivalence classes of $0$ is also
an element of $L^{r,s)}\left( X,\nu \right) ,$ then $\chi _{A}$ is in the
equivalent classes of $0\in L^{r,s)}\left( X,\nu \right) $ with respect to $%
\nu .$ This follows that $\nu \left( A\right) =0.$ Hence, we obtain that $%
\nu \ll \mu .$ One can prove that $\mu \ll \nu $ with the similar method.
This completes the proof.
\end{proof}

\begin{theorem}
\label{subspace}The space $L^{p,q)}\left( X,\mu \right) $ is a subspace of $%
L^{r,s)}\left( X,\nu \right) $ in sense of equivalence classes if and only
if $\mu \approx \nu $ and there is a $C>0$ such that%
\begin{equation}
\left\Vert f\right\Vert _{r,s),\nu }\leq C\left\Vert f\right\Vert _{p,q),\mu
}  \label{equi1}
\end{equation}%
for every $f\in L^{p,q)}\left( X,\mu \right) $.
\end{theorem}

\begin{proof}
Assume that $L^{p,q)}\left( X,\mu \right) \subseteq L^{r,s)}\left( X,\nu
\right) $ in sense of equivalence classes. Moreover, we define the sum norm
on $L^{p,q)}\left( X,\mu \right) $%
\begin{equation*}
\left\Vert \left\vert .\right\vert \right\Vert =\left\Vert .\right\Vert
_{p,q),\mu }+\left\Vert .\right\Vert _{r,s),\nu }.
\end{equation*}%
The space $L^{p,q)}\left( X,\mu \right) $ is a Banach space in sense to $%
\left\Vert \left\vert .\right\vert \right\Vert .$ To prove this, we assume
that $\left( f_{n}\right) _{n\in 
\mathbb{N}
}$ is a Cauchy sequence in $L^{p,q)}\left( X,\mu \right) $. Then for all $%
\beta >0$ there exists $N\left( \beta \right) >0$ whenever $n,m>N\left(
\beta \right) $ such that%
\begin{equation*}
\left\Vert f_{n}-f_{m}\right\Vert _{p,q),\mu }=\underset{0<\varepsilon <q-1}{%
\sup }\left( \frac{q}{p}\varepsilon \tint\limits_{0}^{1}t^{\frac{q}{p}-1}%
\left[ \left( f_{n}-f_{m}\right) ^{\ast }\left( t\right) \right]
^{q-\varepsilon }dt\right) ^{\frac{1}{q-\varepsilon }}<\beta
\end{equation*}%
and%
\begin{equation*}
\left\Vert f_{n}-f_{m}\right\Vert _{r,s),\nu }=\underset{0<\varepsilon <s-1}{%
\sup }\left( \frac{s}{r}\varepsilon \tint\limits_{0}^{1}t^{\frac{s}{r}-1}%
\left[ \left( f_{n}-f_{m}\right) ^{\ast }\left( t\right) \right]
^{s-\varepsilon }dt\right) ^{\frac{1}{s-\varepsilon }}<\beta .
\end{equation*}%
This yields that $\left( f_{n}\right) _{n\in 
\mathbb{N}
}$ is also a Cauchy sequence in $L^{p,q)}\left( X,\mu \right) $ and $%
L^{r,s)}\left( X,\nu \right) $, and $\left( f_{n}\right) _{n\in 
\mathbb{N}
}$ converges to functions $f\in L^{p,q)}\left( X,\mu \right) $ and $g\in
L^{r,s)}\left( X,\nu \right) ,$ respectively. By (\ref{inclusion}), we have%
\begin{equation}
L^{p,q)}\left( X,\mu \right) \hookrightarrow L^{p,q-\varepsilon }\left(
X,\mu \right)  \label{embedd3}
\end{equation}%
for $\varepsilon \in \left( 0,q-1\right) $ and%
\begin{equation}
L^{r,s)}\left( X,\nu \right) \hookrightarrow L^{r,s-\varepsilon }\left(
X,\nu \right)  \label{embedd4}
\end{equation}%
for $\varepsilon \in \left( 0,s-1\right) .$ If we use (\ref{embedd3}), then
we obtain that there is a subsequence $\left( f_{n_{i}}\right) _{i\in 
\mathbb{N}
}$ of $\left( f_{n}\right) _{n\in 
\mathbb{N}
}$ such that $f_{n_{i}}\longrightarrow f$ $\left( \mu \text{-almost
everywhere}\right) $. Moreover, it is easy to prove that $\left(
f_{n_{i}}\right) _{n\in 
\mathbb{N}
}$ converges to $g$ in $L^{q(.),\theta }\left( \nu \right) $ and $%
f_{n_{i}}\longrightarrow g$ $\left( \nu \text{-almost everywhere}\right) $
due to (\ref{embedd4}). Hence, one can find a subsequence $\left(
f_{n_{i_{k}}}\right) $ of $\left( f_{n_{i}}\right) $ such that $%
f_{n_{i_{k}}}\longrightarrow g$ $\left( \nu \text{-almost everywhere}\right) 
$. If we consider the space $L^{p,q)}\left( X,\mu \right) $ is a subspace of 
$L^{r,s)}\left( X,\nu \right) $ in the sense of equivalence classes, then we
have $\mu \approx \nu $ by Theorem \ref{equivaltheo}. This follows that%
\begin{equation*}
\left\vert f(x)-g(x)\right\vert \leq \left\vert
f(x)-f_{n_{i_{k}}}(x)\right\vert +\left\vert
f_{n_{i_{k}}}(x)-g(x)\right\vert ,
\end{equation*}%
that we have $f=g$ $\left( \mu \text{-almost everywhere}\right) $. Since $%
\mu \approx \nu ,$ we obtain $f=g$ $\left( \nu \text{-almost everywhere}%
\right) $, and $f_{n}\longrightarrow f$ in $L^{p,q)}\left( X,\mu \right) $
with respect to the norm $\left\Vert \left\vert .\right\vert \right\Vert .$
Now, we define the unit function $I$ from $\left( L^{p,q)}\left( X,\mu
\right) ,\left\Vert \left\vert .\right\vert \right\Vert \right) $ into $%
\left( L^{p,q)}\left( X,\mu \right) ,\left\Vert .\right\Vert _{p,q),\mu
}\right) .$ Since%
\begin{equation*}
\left\Vert I\left( f\right) \right\Vert _{p,q),\mu }=\left\Vert f\right\Vert
_{p,q),\mu }\leq \left\Vert \left\vert f\right\vert \right\Vert ,
\end{equation*}%
then $I$ is continuous. If we consider the Banach's theorem, then $I$ is a
homeomorphism, see \cite{Car}. This yields the norms $\left\Vert \left\vert
.\right\vert \right\Vert $ and $\left\Vert .\right\Vert _{p,q),\mu }$ are
equivalent. Thus there exists a $C>0$ such that%
\begin{equation*}
\left\Vert f\right\Vert \leq C\left\Vert f\right\Vert _{p,q),\mu }
\end{equation*}%
for all $f\in L^{p,q)}\left( X,\mu \right) $. Finally, we have%
\begin{equation*}
\left\Vert f\right\Vert _{r,s),\nu }\leq \left\Vert f\right\Vert \leq
C\left\Vert f\right\Vert _{p,q),\mu }.
\end{equation*}%
This completes the necessity part of the proof. Now, we suppose that $\mu
\approx \nu $ and the inequality (\ref{equi1}) holds for $L^{p,q)}\left(
X,\mu \right) $. Then, we have $L^{p,q)}\left( X,\mu \right) \subseteq
L^{r,s)}\left( X,\nu \right) $ in the sense of individual functions. By
Theorem \ref{equivaltheo}, the space $L^{p,q)}\left( X,\mu \right) $ is a
subspace of $L^{r,s)}\left( X,\nu \right) $ in the sense of equivalence
classes. If $1<p,r<\infty $ and $q=s=\infty ,$ the proof follows from%
\begin{equation*}
\left\Vert f\right\Vert _{p,\infty ),\mu }=\underset{0<t<1}{\sup }\left[ t^{%
\frac{1}{p}}f^{\ast }\left( t\right) \right] \leq \underset{0<t<1}{\sup }%
\left[ t^{\frac{1}{p}}f^{\ast \ast }\left( t\right) \right] =\left\Vert
f\right\Vert _{p,\infty ),\mu }^{\ast }\leq \frac{p}{p-1}\left\Vert
f\right\Vert _{p,\infty ),\mu }.
\end{equation*}
\end{proof}

\begin{theorem}
The following statements are equivalent.

\begin{enumerate}
\item[\textit{(i)}] The inclusion $L^{p,q)}\left( X,\mu \right) \subseteq
L^{p,q)}\left( X,\nu \right) $ holds.

\item[\textit{(ii)}] $\mu \approx \nu $ and there is a $C>0$ such that%
\begin{equation*}
\nu \left( A\right) \leq C\mu \left( A\right)
\end{equation*}%
for all $A\in \Sigma .$

\item[\textit{(iii)}] We have $L^{1}\left( \mu \right) \subseteq L^{1}\left(
\nu \right) .$
\end{enumerate}
\end{theorem}

\begin{proof}
\textit{(i) }$\Longrightarrow $ \textit{(ii): }By Theorem \ref{subspace}, we
have $\mu \approx \nu $ and there is a $C>0$ such that%
\begin{equation}
\left\Vert f\right\Vert _{p,q),\nu }\leq C_{1}\left\Vert f\right\Vert
_{p,q),\mu }  \label{numu}
\end{equation}%
for all $f\in L^{p,q)}\left( X,\mu \right) .$ In particular, if we consider
the function as $\chi _{A}$ in (\ref{indivi1}) and (\ref{numu}), then we have%
\begin{equation*}
\left( q-1\right) \left( \nu \left( A\right) \right) ^{\frac{1}{p}}\leq
C_{1}\left( q-1\right) \left( \mu \left( A\right) \right) ^{\frac{1}{p}}
\end{equation*}%
This implies%
\begin{equation*}
\nu \left( A\right) \leq C\mu \left( A\right)
\end{equation*}%
where $C=C_{1}^{p}>0.$

\textit{(ii) }$\Longrightarrow $ \textit{(iii): This proof is an immediately
result of }\cite[Proposition 1]{Mi}.

\textit{(iii) }$\Longrightarrow $ \textit{(i): }By the assumption, there
exists a $C_{1}>0$ such that%
\begin{equation}
\left\Vert h\right\Vert _{1,\nu }\leq C_{1}\left\Vert h\right\Vert _{1,\mu }
\label{L1}
\end{equation}%
for all $h\in L^{1}(\mu )$. Let $f\in L^{p,q)}\left( X,\mu \right) $ be
given. Therefore, we obtain%
\begin{equation*}
\left( \frac{q}{p}\varepsilon \tint\limits_{0}^{1}t^{\frac{q}{p}-1}\left(
f^{\ast }\left( t\right) \right) ^{q-\varepsilon }d\mu \left( t\right)
\right) ^{\frac{1}{q-\varepsilon }}\leq \left\Vert f\right\Vert _{p,q),\mu
}<\infty
\end{equation*}%
for all $0<\varepsilon <q-1.$ This yields $\chi _{\left( 0,1\right) }t^{%
\frac{q}{p}-1}\left( f^{\ast }\left( t\right) \right) ^{q-\varepsilon }\in
L^{1}\left( \mu \right) $ for all $0<\varepsilon <q-1.$ By (\ref{L1}), we
get $\chi _{\left( 0,1\right) }t^{\frac{q}{p}-1}\left( f^{\ast }\left(
t\right) \right) ^{q-\varepsilon }\in L^{1}\left( \nu \right) $ and%
\begin{equation*}
\tint\limits_{0}^{1}t^{\frac{q}{p}-1}\left( f^{\ast }\left( t\right) \right)
^{q-\varepsilon }d\nu \left( t\right) \leq C_{1}\tint\limits_{0}^{1}t^{\frac{%
q}{p}-1}\left[ f^{\ast }\left( t\right) \right] ^{q-\varepsilon }d\mu \left(
t\right) .
\end{equation*}%
This implies%
\begin{equation}
\left( \frac{q}{p}\varepsilon \tint\limits_{0}^{1}t^{\frac{q}{p}-1}\left(
f^{\ast }\left( t\right) \right) ^{q-\varepsilon }d\nu \left( t\right)
\right) ^{\frac{1}{q-\varepsilon }}\leq C\left( \frac{q}{p}\varepsilon
\tint\limits_{0}^{1}t^{\frac{q}{p}-1}\left( f^{\ast }\left( t\right) \right)
^{q-\varepsilon }d\mu \left( t\right) \right) ^{\frac{1}{q-\varepsilon }}
\label{L2}
\end{equation}%
where $C=C_{1}^{\frac{1}{q-\varepsilon }}.$ If we take the supremum over $%
0<\varepsilon <q-1$ in (\ref{L2}) , then we have%
\begin{equation*}
\left\Vert f\right\Vert _{p,q),\nu }\leq C\left\Vert f\right\Vert _{p,q),\mu
}<\infty
\end{equation*}%
for every $f\in L^{p,q)}\left( X,\mu \right) .$
\end{proof}

\begin{theorem}
Let $1<q\leq p<r\leq s.$ If the inclusion $L^{p,q)}\left( X,\mu \right)
\subseteq L^{r,s)}\left( X,\mu \right) $ holds, then there is a constant $%
M>0 $ such that $\mu \left( A\right) \geq M$ for every $\mu $-non-null set $%
A\in \Sigma .$
\end{theorem}

\begin{proof}
Assume that $L^{p,q)}\left( X,\mu \right) \subseteq L^{r,s)}\left( X,\mu
\right) .$ By Theorem \ref{subspace}, there is $C>0$ such that%
\begin{equation}
\left\Vert f\right\Vert _{r,s),\mu }\leq C\left\Vert f\right\Vert _{p,q),\mu
}  \label{nonnull1}
\end{equation}%
for all $f\in L^{p,q)}\left( X,\mu \right) .$ Now, let $A\in \Sigma $ be a $%
\mu $-non-null set. Since $p<r,$ we have $\frac{1}{p}-\frac{1}{r}>0.$ In
particular, if we consider the function as $\chi _{A}$ in (\ref{indivi1})
and (\ref{nonnull1}), then we obtain%
\begin{equation*}
\left( s-1\right) \mu \left( A\right) ^{\frac{1}{r}}\leq C\left( q-1\right)
\mu \left( A\right) ^{\frac{1}{p}}
\end{equation*}%
or equivalently 
\begin{equation*}
\frac{s-1}{C\left( q-1\right) }\leq \mu \left( A\right) ^{\frac{1}{p}-\frac{1%
}{r}}.
\end{equation*}%
If we set $M=\left( \frac{s-1}{C\left( q-1\right) }\right) ^{\frac{rp}{r-p}%
}, $ we get $\mu \left( A\right) \geq M.$
\end{proof}

\section{Inclusion Theorems of $\Lambda _{p),\protect\omega }$}

In \cite{Jain}, the authors unified the concepts of the classical Lorentz
space and the grand Lebesgue space that would result in a new space $\Lambda
_{p),\omega }$ denoted as follows$.$

\begin{definition}
Assume that $1<p<\infty $ and $\omega $ is a weight function. The grand
Lorentz space $\Lambda _{p),\omega }$ is to be space of all measurable
functions $f$ defined on $\left( 0,1\right) $ such that%
\begin{equation*}
\left\Vert f\right\Vert _{\Lambda _{p),\omega }}=\underset{0<\varepsilon <p-1%
}{\sup }\left( \varepsilon \tint\limits_{0}^{1}\left( f^{\ast }\left(
t\right) \right) ^{p-\varepsilon }\omega \left( t\right) dt\right) ^{\frac{1%
}{p-\varepsilon }}<\infty .
\end{equation*}%
The space $\Lambda _{p),\omega }$ is a BF-space when $\omega $ is decreasing
integrable weight. Moreover, the following inclusion holds for all $%
0<\varepsilon <p-1:$%
\begin{equation}
\Lambda _{p,\omega }\hookrightarrow \Lambda _{p),\omega }\hookrightarrow
\Lambda _{p-\varepsilon ,\omega }.  \label{inclusion}
\end{equation}%
Also, if we consider $\omega \left( t\right) =\frac{q}{p}t^{\frac{q}{p}-1},$
then we can say that the space $L^{p,q)}\left( X,\mu \right) $ is a special
case of $\Lambda _{p),\omega },$ see \cite{Jain}.
\end{definition}

It is easy to see that $\Lambda _{p),\vartheta }\hookrightarrow \Lambda
_{p),\omega }$ if $\omega \prec \vartheta .$ Moreover, as a direct result of
the definition of $\omega \sim \vartheta ,$ we obtain $\Lambda
_{p),\vartheta }=\Lambda _{p),\omega }.$ Now, we will present several
advanced inclusion theorems for grand Lorentz spaces.

\begin{theorem}
\label{teokucukbuyuk}Let $1<p\leq q<\infty .$ Then we have $\Lambda
_{p),\omega }\hookrightarrow \Lambda _{q),\omega }$ if and only if%
\begin{equation}
W\left( 1\right) ^{\frac{1}{q-\varepsilon }-\frac{1}{p-\varepsilon }}<\infty
\label{Wholds}
\end{equation}%
holds for $0<\varepsilon <p-1.$
\end{theorem}

\begin{proof}
The necessary part is trivial by substituting in the inequality $f=\chi _{%
\left[ 0,1\right] }$ where $\chi _{A}\left( x\right) =1$ for $x\in A$ and $%
\chi _{A}\left( x\right) =0$ for $x\notin A.$ Indeed., since the functions $%
\varepsilon ^{\frac{1}{p-\varepsilon }}$ and $\varepsilon ^{\frac{1}{%
q-\varepsilon }}$ are increasing for $0<\varepsilon <p-1<q-1,$ we obtain%
\begin{eqnarray*}
\left( q-1\right) \underset{0<\varepsilon <q-1}{\sup }\left(
\tint\limits_{0}^{1}\omega \left( t\right) dt\right) ^{\frac{1}{%
q-\varepsilon }} &\leq &C\left( p-1\right) \underset{0<\varepsilon <p-1}{%
\sup }\left( \tint\limits_{0}^{1}\omega \left( t\right) dt\right) ^{\frac{1}{%
p-\varepsilon }} \\
&\leq &C\left( p-1\right) \underset{0<\varepsilon <q-1}{\sup }\left(
\tint\limits_{0}^{1}\omega \left( t\right) dt\right) ^{\frac{1}{%
p-\varepsilon }}.
\end{eqnarray*}%
This yields that%
\begin{equation*}
\underset{0<\varepsilon <q-1}{\sup }\left( \tint\limits_{0}^{1}\omega \left(
t\right) dt\right) ^{\frac{1}{q-\varepsilon }-\frac{1}{p-\varepsilon }}\leq
C<\infty .
\end{equation*}%
Now, let $f\in \Lambda _{p),\omega }$ be given. By (\ref{inclusion}), we get 
$f\in \Lambda _{p-\varepsilon ,\omega }.$ Since (\ref{Wholds}) holds, we
have $\Lambda _{p-\varepsilon ,\omega }\hookrightarrow \Lambda
_{q-\varepsilon ,\omega }$ for $0<\varepsilon <p-1$, see \cite[Theorem 3.1]%
{Carro}. This yields that there exists $C\left( \varepsilon \right) >0$ such
that%
\begin{equation}
\left\Vert f\right\Vert _{\Lambda _{q-\varepsilon ,\omega }}\leq C\left(
\varepsilon \right) \left\Vert f\right\Vert _{\Lambda _{p-\varepsilon
,\omega }}  \label{KappaInc}
\end{equation}%
for $f\in \Lambda _{p-\varepsilon ,\omega }$ and $\varepsilon \in \left(
0,p-1\right) $. It is note that identity operator does not exceed $\mu
\left( X\right) +1$, see \cite{Kor}. By (\ref{KappaInc}), we have%
\begin{equation*}
\left( \varepsilon \tint\limits_{0}^{1}\left( f^{\ast }\left( t\right)
\right) ^{q-\varepsilon }\omega (t)dt\right) ^{\frac{p-\varepsilon }{%
q-\varepsilon }}\leq 2^{p-\varepsilon }\varepsilon ^{\frac{p-\varepsilon }{%
q-\varepsilon }}\tint\limits_{0}^{1}\left( f^{\ast }\left( t\right) \right)
^{p-\varepsilon }\omega (t)dt.
\end{equation*}%
This obtain%
\begin{equation*}
\left( \varepsilon \tint\limits_{0}^{1}\left( f^{\ast }\left( t\right)
\right) ^{q-\varepsilon }\omega (t)dt\right) ^{\frac{1}{q-\varepsilon }}\leq
2\varepsilon ^{\frac{1}{q-\varepsilon }}\varepsilon ^{-\frac{1}{%
p-\varepsilon }}\left( \varepsilon \tint\limits_{0}^{1}\left( f^{\ast
}\left( t\right) \right) ^{p-\varepsilon }\omega (t)dt\right) ^{\frac{1}{%
p-\varepsilon }}
\end{equation*}%
or equivalently%
\begin{equation}
\varepsilon ^{\frac{q-p}{\left( p-\varepsilon \right) \left( q-\varepsilon
\right) }}\left( \varepsilon \tint\limits_{0}^{1}\left( f^{\ast }\left(
t\right) \right) ^{q-\varepsilon }\omega (t)dt\right) ^{\frac{1}{%
q-\varepsilon }}\leq 2\left( \varepsilon \tint\limits_{0}^{1}\left( f^{\ast
}\left( t\right) \right) ^{p-\varepsilon }\omega (t)dt\right) ^{\frac{1}{%
p-\varepsilon }}.  \label{KappaInc2}
\end{equation}%
If we take the supremum over $0<\varepsilon <p-1<q-1$ in both sides of (\ref%
{KappaInc2}), then we get%
\begin{equation*}
\left\Vert f\right\Vert _{\Lambda _{q),\omega }}\leq C\left\Vert
f\right\Vert _{\Lambda _{p),\omega }}
\end{equation*}%
where $\underset{0<\varepsilon <q-1}{\sup }\varepsilon ^{\frac{q-p}{\left(
p-\varepsilon \right) \left( q-\varepsilon \right) }}<\infty .$ This
completes the proof.
\end{proof}

By the similar method in Theorem \ref{teokucukbuyuk}, we have

\begin{theorem}
Let $1<p\leq q<\infty .$ Then we have $\Lambda _{p),\vartheta
}\hookrightarrow \Lambda _{q),\omega }$ if and only if%
\begin{equation*}
W\left( 1\right) ^{\frac{1}{q-\varepsilon }}V\left( 1\right) ^{-\frac{1}{%
p-\varepsilon }}<\infty
\end{equation*}%
holds for $0<\varepsilon <p-1.$
\end{theorem}

\begin{theorem}
Let $1<q<p<\infty $ and let $r$ be given by $\frac{1}{r}=\frac{1}{q}-\frac{1%
}{p}.$ Then we have $\Lambda _{p),\vartheta }\hookrightarrow \Lambda
_{q),\omega }$ if%
\begin{eqnarray}
&&\left( \tint\limits_{0}^{\infty }\left( \frac{W\left( t\right) }{V\left(
t\right) }\right) ^{\frac{r-\varepsilon }{p-\varepsilon }}\omega \left(
t\right) dt\right) ^{\frac{1}{r-\varepsilon }}  \notag \\
&=&\left( \frac{q-\varepsilon }{r-\varepsilon }\frac{W^{\frac{r-\varepsilon 
}{q-\varepsilon }}\left( \infty \right) }{V^{\frac{r-\varepsilon }{%
p-\varepsilon }}\left( \infty \right) }+\frac{q-\varepsilon }{p-\varepsilon }%
\tint\limits_{0}^{\infty }\left( \frac{W\left( t\right) }{V\left( t\right) }%
\right) ^{\frac{r-\varepsilon }{q-\varepsilon }}\vartheta \left( t\right)
dt\right) ^{\frac{1}{r-\varepsilon }}<\infty  \label{buyukkucuk}
\end{eqnarray}%
holds for $0<\varepsilon <q-1.$
\end{theorem}

\begin{proof}
Let $f\in \Lambda _{p),\vartheta }$ be given. By (\ref{inclusion}), we get $%
f\in \Lambda _{p-\varepsilon ,\vartheta }.$ Moreover, we assume that (\ref%
{buyukkucuk}) holds. Thus, we have $\Lambda _{p-\varepsilon ,\vartheta
}\hookrightarrow \Lambda _{q-\varepsilon ,\omega }$ for $0<\varepsilon <q-1$%
, see \cite[Theorem 3.1]{Carro}. This yields that there exists $C\left(
\varepsilon \right) >0$ such that%
\begin{equation}
\left\Vert f\right\Vert _{\Lambda _{q-\varepsilon ,\omega }}\leq C\left(
\varepsilon \right) \left\Vert f\right\Vert _{\Lambda _{p-\varepsilon
,\vartheta }}  \label{bykkck1}
\end{equation}%
for $f\in \Lambda _{p-\varepsilon ,\vartheta }$ and $\varepsilon \in \left(
0,p-1\right) $. It is note that identity operator does not exceed $\mu
\left( X\right) +1$, see \cite{Kor}. By (\ref{bykkck1}), we have%
\begin{equation*}
\left( \varepsilon \tint\limits_{0}^{1}\left( f^{\ast }\left( t\right)
\right) ^{q-\varepsilon }\omega (t)dt\right) ^{\frac{p-\varepsilon }{%
q-\varepsilon }}\leq 2^{p-\varepsilon }\varepsilon ^{\frac{p-\varepsilon }{%
q-\varepsilon }}\tint\limits_{0}^{1}\left( f^{\ast }\left( t\right) \right)
^{p-\varepsilon }\vartheta (t)dt.
\end{equation*}%
This obtain%
\begin{equation}
\left( \varepsilon \tint\limits_{0}^{1}\left( f^{\ast }\left( t\right)
\right) ^{q-\varepsilon }\omega (t)dt\right) ^{\frac{1}{q-\varepsilon }}\leq
2\varepsilon ^{\frac{p-q}{\left( p-\varepsilon \right) \left( q-\varepsilon
\right) }}\left( \varepsilon \tint\limits_{0}^{1}\left( f^{\ast }\left(
t\right) \right) ^{p-\varepsilon }\vartheta (t)dt\right) ^{\frac{1}{%
p-\varepsilon }}.  \label{bykkck2}
\end{equation}%
If we take the supremum over $0<\varepsilon <q-1<p-1$ in both sides of (\ref%
{bykkck2}), then we get%
\begin{equation*}
\left\Vert f\right\Vert _{\Lambda _{q),\omega }}\leq C\left\Vert
f\right\Vert _{\Lambda _{p),\vartheta }}
\end{equation*}%
where $\underset{0<\varepsilon <q-1}{\sup }\varepsilon ^{\frac{p-q}{\left(
p-\varepsilon \right) \left( q-\varepsilon \right) }}<\infty .$ This
completes the proof.
\end{proof}

\begin{corollary}
Since the space $L^{p,q)}\left( X,\mu \right) $ is a special case of $%
\Lambda _{p),\omega },$ the following inclusion follows (\ref{Wholds}) that%
\begin{equation*}
L^{p,q_{1})}\subset L^{p)}\left( \mu \right) \subset L^{p,s_{1})}
\end{equation*}%
for $q_{1}\leq p\leq s_{1}.$ In similar way, we have%
\begin{equation*}
L^{p,q_{2})}\subset L^{p)}\left( \mu \right) \subset L^{p,s_{2})}
\end{equation*}%
under the condition (\ref{buyukkucuk}) for $s_{2}\leq p\leq q_{2}.$
\end{corollary}

\section{Approximate Identities in $\Lambda _{p),\protect\omega }$}

Let $\Omega \subset 
\mathbb{R}
^{d}$ be bounded and open set. It is well known that the classical Lebesgue
space $L^{p}(\Omega )$ has a bounded approximate identity in $L^{1}(\Omega )$%
. Gurkanli considered $L^{p),\theta }\left( \Omega \right) $ does not admit
a bounded approximate identity in $L^{1}(\Omega )$ in \cite[Theorem 4]{Gur},
and also $\left[ L^{p}\left( \Omega \right) \right] _{p),\theta }$, the
closure of $C_{0}^{\infty }(\Omega )$ in $L^{p),\theta }\left( \Omega
\right) $, admits a bounded approximate identity in $L^{1}(\Omega )$ in \cite%
[Theorem 6]{Gur}. It is known that when $\theta =1$ the space $%
L^{p),1}\left( \Omega \right) $ reduces to the grand Lebesgue space $%
L^{p)}\left( \Omega \right) .$

For $x\in X$ and $r>0,$ we denote an open ball with center $x$ and radius $r$
by $B(x,r)$. For $f\in L_{loc}^{1}\left( \Omega \right) ,$ we denote the
(centered) Hardy-Littlewood maximal operator $Mf$\ of $f$ by%
\begin{equation*}
Mf(x)=\underset{r>0}{\sup }\frac{1}{\left\vert B(x,r)\right\vert }%
\int\limits_{B(x,r)}\left\vert f(y)\right\vert dy,
\end{equation*}%
where the supremum is taken over all balls $B(x,r).$ In \cite{Jain}, the
authors proved that the Hardy-Littlewood maximal operator is bounded in $%
\Lambda _{p),\vartheta }$.

\begin{definition}
Assume that $\varphi $ is an integrable function defined on $%
\mathbb{R}
^{d}$ such that $\int\limits_{%
\mathbb{R}
^{d}}\varphi (x)dx=1.$ For each $t>0,$ define the function $\varphi
_{t}\left( x\right) =t^{-d}\varphi \left( \frac{x}{t}\right) .$ The sequence 
$\left\{ \varphi _{t}\right\} $ is referred to as an approximate identity.
It is known that for $1<p<\infty ,$ the sequence $\left\{ \varphi _{t}\ast
f\right\} $ converges to $f$ in $L^{p}\left( \Omega \right) $, i.e.%
\begin{equation*}
\lim_{t\longrightarrow \infty }\left\Vert \varphi _{t}\ast f-f\right\Vert
_{p,\Omega }=0
\end{equation*}%
,see \cite{St}. If we impose additional conditions on $\varphi $, then the
entire sequence converges almost everywhere to $f.$ Define the radial
majorant of $\varphi $ to be the function%
\begin{equation*}
\widetilde{\varphi }(x)=\underset{\left\vert y\right\vert \geq \left\vert
x\right\vert }{\sup }\left\vert \varphi (y)\right\vert .
\end{equation*}%
If the function $\widetilde{\varphi }$ is integrable, then $\left\{ \varphi
_{t}\right\} $ is called a potential-type approximate identity, see \cite{Cu}%
.
\end{definition}

\begin{theorem}
\label{supremummaximal}(see \cite{Duo})Let $f\in L_{loc}^{1}\left( X\right) $
be given. Then we have%
\begin{equation*}
\underset{t>0}{\sup }\left\vert \varphi _{t}\ast f\left( x\right)
\right\vert \leq Mf\left( x\right) .
\end{equation*}
\end{theorem}

\begin{theorem}
\label{Lorentzgirisim}If $f\in \Lambda _{p,\omega },$ then we have $\varphi
_{t}\ast f\longrightarrow f$ in $\Lambda _{p,\omega }$ as $t>0.$
\end{theorem}

\begin{proof}
Let $f\in \Lambda _{p,\omega }$ and $\varepsilon >0$ be given. By the
Theorem \ref{supremummaximal} and the boundedness of maximal function in $%
\Lambda _{p),\omega }$, we have%
\begin{equation*}
\left\Vert \varphi _{t}\ast f\right\Vert _{\Lambda _{p,\omega }}\leq C^{\ast
}\left\Vert Mf\right\Vert _{\Lambda _{p,\omega }}\leq C\left\Vert
f\right\Vert _{\Lambda _{p,\omega }}<\infty
\end{equation*}%
and then $\varphi _{t}\ast f\in \Lambda _{p,\omega }$ for all $t>0.$ If we
use the similar method in \cite{Yap}, then it is easy to see that the space $%
C_{c}\left( X\right) $ is dense in $\Lambda _{p),\omega }$. Therefore, there
exists a function $g\in C_{c}\left( X\right) $ such that%
\begin{equation}
\left\Vert f-g\right\Vert _{\Lambda _{p,\omega }}<\varepsilon  \label{1}
\end{equation}%
as $t>0.$ Moreover, for all $t>0,$ we have $\varphi _{t}\ast g\in
C_{0}^{\infty }\left( X\right) ,$ see \cite[Theorem 2.29]{Ad}. It is easily
seen that $\varphi _{t}\ast g\longrightarrow g$ uniformly on compact sets as 
$t\longrightarrow 0^{+}$. Hence we have%
\begin{equation}
\left\Vert \varphi _{t}\ast g-g\right\Vert _{\Lambda _{p,\omega
}}<\varepsilon .  \label{2}
\end{equation}

Finally by using (\ref{1}) and (\ref{2})$,$%
\begin{eqnarray*}
\left\Vert f-\varphi _{t}\ast f\right\Vert _{\Lambda _{p,\omega }} &\leq
&\left\Vert f-g\right\Vert _{\Lambda _{p,\omega }}+\left\Vert g-\varphi
_{t}\ast g\right\Vert _{\Lambda _{p,\omega }}+\left\Vert \varphi _{t}\ast
g-\varphi _{t}\ast f\right\Vert _{\Lambda _{p,\omega }} \\
&<&\varepsilon .
\end{eqnarray*}%
This completes the proof.
\end{proof}

Now, we are ready to present the main theorem of this section for the space $%
\Lambda _{p,\omega }.$

\begin{theorem}
Let $\left\{ \varphi _{t}\right\} $ be a potential-type approximate
identity. Then If $p<\infty $, then we have $\left\Vert \varphi _{t}\ast
f-f\right\Vert _{\Lambda _{p),\omega }}\longrightarrow 0$ as $%
t\longrightarrow 0^{+}$ for $f\in \Lambda _{p),\omega }.$ Moreover, we get%
\begin{equation*}
\left\Vert \varphi _{t}\ast f\right\Vert _{\Lambda _{p),\omega }}\leq
C\left( A,p\right) \left\Vert Mf\right\Vert _{\Lambda _{p),\omega }}\leq
C\left( A,p\right) \left\Vert f\right\Vert _{\Lambda _{p),\omega }}.
\end{equation*}
\end{theorem}

\begin{proof}
Let $f\in \Lambda _{p),\omega }$ be given. By Theorem \ref{Lorentzgirisim},
for every $\eta >0$ there exists an $h>0$ such that%
\begin{equation*}
\left\Vert \varphi _{t}\ast f-f\right\Vert _{\Lambda _{p-\varepsilon ,\omega
}}<\eta
\end{equation*}%
for all $t$ satisfying $t<h$. This follows that 
\begin{equation*}
\left\Vert \varphi _{t}\ast f-f\right\Vert _{\Lambda _{p),\omega }}<\eta
\sup_{0<\varepsilon <q-1}\varepsilon ^{\frac{1}{q-\varepsilon }}=\left(
q-1\right) \eta .
\end{equation*}%
That is the desired result.
\end{proof}

\section{Acknowledgment}

We express our thanks to Professor Amiran Gogatishvili for kind comments and
helpful suggestions.

\end{document}